\documentclass[12pt,a4paper]{article}
\usepackage{amssymb,amsmath,amsthm}

\usepackage{graphicx}
\usepackage{epsf}

\newtheorem{theorem}{Theorem}[]

\renewcommand{\geq}{\geqslant}
\renewcommand{\leq}{\leqslant}

\newcommand{\R}{\mathbb{R}}
\newcommand{\eps}{\varepsilon}
\newcommand{\PP}{\mathbb{P}}
\newcommand{\E}{\mathbb{E}}
\newcommand{\e}{\mathrm{e}}
\newcommand{\stochconv}{\stackrel{{\mathcal{D}}}{\longrightarrow}}

\newcommand{\tX}{\widetilde{X}}
\newcommand{\tZ}{\widetilde{Z}}
\newcommand{\tro}{\tilde{\rho}}

\def\d{{\rm d}}
\def\R{{\mathbb{R}}}

\begin{document}

\begin{center}
\textbf{ \Large{The diminishing segment process}}
\bigskip

\textsc{Gergely Ambrus
\footnote{Supported by the OTKA grants 75016 and 76099.}} \\
Alfr\'ed R\'enyi Institute of Mathematics, Hungarian Academy of Sciences, PO Box
127, 1364 Budapest, Hungary; \textit{e-mail:} \texttt{ambrus@renyi.hu}

\smallskip

\textsc{P\'eter Kevei \footnote{Supported by the Analysis and Stochastics
Research Group of the Hungarian Academy of Sciences.}} \\
Centro de Investigaci\'on en Matem\'aticas, Jalisco S/N, Valenciana, Guanajuato,
GTO 36240, Mexico; \textit{e-mail:} \texttt{kevei@cimat.mx}
\smallskip

\textsc{Viktor V\'igh \footnote{Supported by the OTKA grant 75016, and by NSERC of Canada.}} \\
Department of Mathematics and Statistics, University of Calgary\\ 2500
University Drive NW Calgary, AB, Canada
T2N 1N4\\
\textit{e-mail:} \texttt{vigvik@gmail.com}\\

\end{center}

\begin{abstract}
Let $\Xi_0 = [-1,1]$, and define the segments $\Xi_n$ recursively in the
following manner: for every $n=0, 1, \dots$, let $\Xi_{n+1} = \Xi_n \cap
[a_{n+1} -1, a_{n+1} +1]$, where the point $a_{n+1}$ is chosen randomly on the
segment $\Xi_n$ with uniform distribution. For the radius $\rho_n$ of $\Xi_n$
we prove that $n (\rho_n - 1/2)$ converges in distribution to an exponential
law, and we show that the centre of the limiting unit interval  has arcsine
distribution.

\end{abstract}

\section{Introduction} \label{introd}

We consider the following stochastic process. Let $\Xi_0 = [-1,1]$, and
define the segments $\Xi_n$ recursively in the following manner: for every
$n=0, 1, \dots$, let
\[
\Xi_{n+1} = \Xi_n \cap [a_{n+1} -1, a_{n+1} +1],
\]
where the point $a_{n+1}$ is chosen randomly on the segment $\Xi_n$ with uniform
distribution.
After $n$ steps, one obtains the segment
\[
\Xi_n = [Z_n - \rho_n, Z_n + \rho_n].
\]
The (centre, radius) process $(Z_n, \rho_n)$ is a continuous state space Markov
chain.  The radius sequence $(\rho_n)$ is monotonically decreasing, and  it is
easy to see that with probability 1, $\rho_n \to 1/2$. Moreover, $(Z_n)$ is a
convergent sequence, assuming values on $[-1/2, 1/2]$. Denote $\lim_{n\to \infty} Z_n$ by
$Z$.

We are interested in the most straightforward
questions: {\em
\begin{itemize}
\item[(1)] What is the distribution of $Z_n$ and $\rho_n$ for a given $n$?
\item[(2)] What is the asymptotic behaviour of the radius?
\item[(3)] What is the limit distribution of the centre?
\end{itemize}
}

Our work was motivated by the following problem formulated by B\'alint
T\'oth (T\'oth, 2010) some 20 years ago with $K$ being the unit disc of the plane. 
Let $K = K_0$ be a convex body
in the $\R^d$ that contains the origin, and  define the process $(K_n, p_n)$,
$n \geq 1$, similarly to the above construction: let $p_{n+1}$ be a uniform
random point in $K_n$, and set $K_{n+1} = K_n \cap (p_{n+1} + K)$. Clearly,
$(K_n)$ is a nested sequence of convex bodies which converge to a non-empty
limit object, again a convex body in $\R^d$. What can we say about the
distribution of this limit body?
When $K$ is the unit disc the limit object is almost surely a convex disc of constant
width 1. The present note deals with the $1$-dimensional analogue of this
problem. Intriguingly, apart from almost trivial results, nothing is known
about the related questions even in the plane.

Another direction of generalising the problem treated here is to choose the
subsequent centres according to a previously fixed distribution instead of the
uniform one. Research on this version is currently in progress.

The paper is organised as follows. In Section 2 we give a recursion for the
density function of $\rho_n$, which allows us to explicitly calculate the
expectation for small $n$'s. In Section 3 we show that $n (\rho_n - 1/2)$
converges in distribution to an exponential law, which actually shows the
rapidness of the process. Section 4 contains the limit distribution of the
centre, while in the last section we derive a somewhat unexpected connection
between the process and the Poisson--Dirichlet distribution.

\section{First observations}

At the $(n+1)$st step, we separate two cases according to the location of
$a_{n+1}$. First, if $a_{n+1} \in [Z_n + \rho_n -1 , Z_n - \rho_n + 1]$, then
no change occurs to the segment, and thus $\Xi_{n+1} = \Xi_n$. In the second
case, when $a_{n+1}$ is close to one of the endpoints of $\Xi_n$, the centre
moves and the length decreases. Introduce yet two new random processes
measuring the change of the location of the centre by
\[
\eps_n X_n = Z_n  -Z_{n-1}, \ n \geq 1,
\]
with $\eps_n = \pm 1$ and $X_n \geq 0$ (if $Z_{n} = Z_{n-1}$, then of no
consequence, let $\eps_n = 1$). Thus, for $n \geq 1$,
\begin{equation} \label{x-def}
X_{n} =
\begin{cases}
0 & \textrm{w.prob.  $(1 - \rho_{n-1})/\rho_{n-1} $}\\
\textrm{uniform on $[0,\rho_{n-1} - 1/2]$} & \textrm{w.prob. $(2 \rho_{n-1} -
1)/\rho_{n-1}$,}
\end{cases}
\end{equation}
moreover,
\[
\PP(\eps_n=1 | X_n \neq 0 ) = \frac 1 2 \, .
\]
By definition,
\begin{equation}\label{znx}
Z_n =\sum_{i=1}^n \eps_i X_i,
\end{equation}
and it is easy to see that
\begin{equation*}\label{rx}
\rho_n = 1 - \sum_{i=1}^n X_i =: 1 - S_n.
\end{equation*}
Thus, with probability 1, $\sum_1^\infty X_i = 1/2$.  By an inductive argument,
it follows that $S_n$ has a continuous distribution. Denote by $f_n(x)$ the
probability density function of $S_n$. Using the Markov property and
\eqref{x-def} it is easy to express $f_{n+1}$ in terms of $f_n$: for $0 \leq x
\leq 1/2$,
\begin{align*}
\PP(S_{n+1} > x)
& = \int_0^{1/2} \PP(  S_{n+1} > x | S_n = y ) f_n(y) \, \d y \\
& = \PP(S_n  >  x) + \int_0^x f_n(y) \frac{1-2x}{1-y} \, \d y,
\end{align*}
whence differentiating
\begin{equation}\label{fnrecurs}
f_{n+1}(x) = f_n(x) \frac  x {1-x} + 2 \int_0^x \frac {f_n(y)}{1 -y}  \, \d y.
\end{equation}
The first few examples are, for $0 \leq x \leq 1/2$,
\begin{align*}
f_1(x) &= 2 \\
f_2(x) &= \frac{2 x}{1 - x} - 4 \ln(1 - x)\\
f_3(x) &= \frac{2 x(2-x)}{(1 - x)^2} + \frac{4 (1 - 2 x)}{ 1 - x} \ln (1-x) + 4
( \ln (1-x))^2.
\end{align*}

In order to calculate the expectation of $S_n$ (and thus $\E \rho_n$), we
consider the Taylor series expansion of $f_n$ about 0:
\begin{equation*}
f_n(x) = \sum_{k=0}^\infty c_{n,k}x^k.
\end{equation*}
Based on \eqref{fnrecurs}, one readily obtains the formula
\begin{equation}\label{cnrecurs}
c_{n+1,k} = \frac {k+2}{k} \sum_{j = 0}^{k-1} c_{n,j}\, ,
\end{equation}
and hence
\begin{equation} \label{esnrecurs} \E S_n = \int_{0}^{1/2} x f_n(x)
\d x = \sum_{k=0}^\infty c_{n,k} \frac {2^{-(k+2)}}{k+2}  =\frac 1 4  \sum
_{j=0}^\infty c_{n-1,j} \sum_{k = j+1}^\infty \frac {2^{-k}}{k},
\end{equation}
which may be useful for obtaining a recursive expression involving only the
expectations. These formulas also enable us to efficiently compute the
expectations for relatively small $n$'s, see the figure below.

\begin{figure}[h]
\epsfxsize =8 cm \centerline{\epsffile{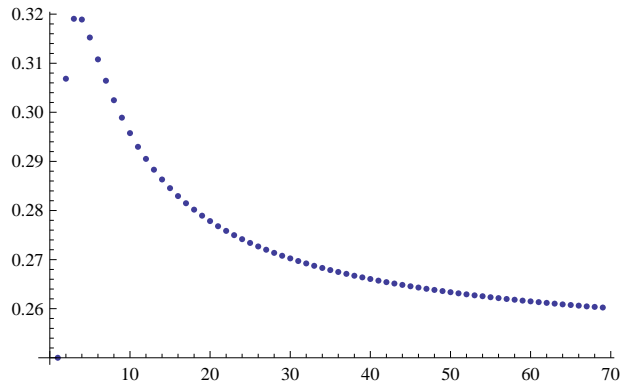}}
  \caption{The first 70 values of $n (1/2 - \E S_n)$ }
 \label{fv1}
\end{figure}

\section{Asymptotics of the length}

We determine the asymptotic behavior of $\rho_n$, the radius of the $n$th
segment.

\begin{theorem}
We have the distributional convergence
\[
n \left( \rho_n - \frac{1}{2} \right) \stochconv \mathrm{Exp}(4).
\]
Accordingly, for any $k \geq 1$,
\[
\lim_{n \to \infty} \E n^k \left( \rho_n - \frac{1}{2} \right)^k =
\frac{k!}{4^k}\,.
\]
\end{theorem}

\begin{proof}
Since $\rho_n = 1 - S_n$, it is equivalent to prove the corresponding limit theorems for
$1/2 - S_n$.

It follows from (\ref{x-def}) that
$$
S_{n+1} =
\begin{cases}
S_n & \textrm{w.prob.  $ S_n / (1 -  S_n)$}\\
\textrm{uniform on $[S_n, 1/2]$} & \textrm{w.prob. $(1 - 2 S_n) / (1 - S_n)$.}
\end{cases}
$$
Observe that given $S_n$, $S_{n+1}$
has the same distribution as $ \max \{ S_n, U \}$, where $U$ is a uniform
random variable on $[S_n /2,1/2]$, independent of $S_n$. Let $0 \leq \alpha
\leq 1/2$, and $U_1^\alpha, U_2^\alpha, \dots$ be independent, uniform random
variables on $[1/4 - \alpha/2, 1/2]$. Define $M^\alpha_n = \max_{1 \leq i \leq
n} U^\alpha _i$. It is well-known in extreme value theory (Billingsley, 1995, p.192), and
easy to check that
\begin{equation} \label{maxconv}
n \left(  \frac{1}{2} - M_n^\alpha  \right) \stochconv \mathrm{Exp}(4/(1+ 2 \alpha)).
\end{equation}

We say that $X$ \emph{stochastically dominates} (or just dominates) $Y$ if
$\PP( X > x) \geq \PP(Y > x)$ for all $x \in \R$.

To obtain a lower estimate, we use that $S_n \leq 1/2$, and thus an easy
induction argument shows that $M^0_n$ dominates $S_n$. From this and
\eqref{maxconv} we obtain
\[
\PP \left( n \left( \frac{1}{2} - S_n \right) > x \right) \geq \PP \left( n
\left( \frac{1}{2} - M_n^0 \right) > x \right) \to \e^{-4x},
\]
i.e.
\[
\liminf_{n \to \infty} \PP \left( n \left( \frac{1}{2} - S_n \right) > x
\right) \geq \e^{-4x}, \quad x \geq 0.
\]

The almost sure convergence $S_n \to 1/2$ heuristically means that with
overwhelming probability, $1/2 - S_n/2 \approx 1/4$,  hence for $n$
sufficiently large, $S_n$ behaves approximately like $M^0_n$ does. This
heuristic idea is made precise as follows. Fix the small positive numbers
$\beta
> 0$ and $\eps > 0$. Since $S_n \to 1/2$ a.s., for $n$ sufficiently large,
$\PP( S_{\beta n} < 1/2 - \eps ) \leq \eps$. Moreover, if $S_{\beta n} \geq 1/2
- \eps$, then $S_n$ is minored by $M^\eps_{n - \beta n}$. This can be shown
again by induction. Thus, by \eqref{maxconv},
\begin{align*}
\PP \left( n \left( \frac{1}{2} - S_n \right) > x \right)
& = \PP \left( S_n < \frac{1}{2} - \frac{x}{n} \right)  \\
& \leq \PP \left( S_n < \frac{1}{2} - \frac{x}{n} \,
\Big| \, S_{\beta n} \geq \frac{1}{2} - \eps  \right) + \eps \\
& \leq \PP \left( M_{( 1 - \beta) n}^\eps < \frac{1}{2} - \frac{x}{n} \right) + \eps\\
& \to \e^{- \frac{4x(1-\beta)}{1 + 2 \eps}} + \eps \, ,
\end{align*}
i.e.~for $x \geq 0$
\[
\limsup_{n \to \infty} \PP \left( n \left( \frac{1}{2} - S_n \right) > x
\right) \leq \e^{- \frac{4 x(1-\beta)}{1 + 2 \eps}} + \eps.
\]
Since $\eps$ and $\beta$ are arbitrary, this gives the distributional
convergence.

To prove the convergence of the moments, it is enough to show that for any $k
\geq 1$, the sequence $\{ n^k \left( \frac{1}{2} - S_n \right)^k
\}_{n=1}^\infty$ is uniformly integrable (Billingsley, 1995, Theorem 25.12). The fact
that $S_n$ dominates $M_n^{1/2}$ readily implies
\[
 \PP \left( n \left( \frac{1}{2} - S_n \right) > x \right) \leq \e^{-2x},
\]
which shows  uniform integrability.
\end{proof}

In particular, $\E \rho_n = 1/2 - 1/(4n) + o(n^{-1})$ and $\mathrm{Var}(\rho_n)
\sim 1 / (16 n^2)$. The question of obtaining an exact or asymptotic formula
for $\E S_n$ using only the recursive relations \eqref{cnrecurs} and
\eqref{esnrecurs} remains open.

\section{Limit distribution}

In this section, we determine the limit behaviour of $Z_n$. Let $F(x)$ denote
the cumulative distribution function of $Z$; based on the definition of $Z$,
it follows that $F(x) = 0$ if $x \leq -1/2$
and $F(x) = 1$ if $x \geq 1/2$.

\begin{theorem} \label{center}
The distribution of $Z$ is a translated arcsine law: for $-1/2 \leq x \leq
1/2$,
\begin{equation*} \label{arcsin-df}
F(x) = \frac 2 \pi \arcsin \sqrt{x + 1/2} \, .
\end{equation*}
\end{theorem}

\begin{proof}
By \eqref{znx}, we have to determine the limit distribution of $\sum \eps_i
X_i$. Clearly, this is not affected by the steps where $X_i=0$. We introduce
the thinned process $(\tZ_n,\tro_n)$ as follows: for $n \geq 1$, let $\xi_n$ be
independent Bernoulli(1/2) random variables, and $\tX_n$ be a uniform random
variable on $[0, 1/2 - \sum_1^{n-1} \tX_i]$. The centre of the segment after
the $n$th step of the thinned process is given by $\tZ_n = \sum_1^n \xi_i
\tX_i$, and the radius is $\tro_n= 1 - \sum_1^n \tX_i$. Plainly,
\[
Z \stackrel{\mathcal{D}}{=}  \sum_{i=1}^\infty \xi_i \tX_i.
\]
Introduce $r_n = 2 \tro_n - 1$. If $U_1, U_2, \dots$ are i.i.d. Uniform$(0,1)$
random variables, then setting $(\tZ_0, r_0) = (0,1)$,

\begin{equation} \label{zrepr}
\begin{split}
\tZ_{n+1} & =  \tZ_n + \frac{1}{2}\, \xi_{n+1} (1 - U_{n+1})  r_n \, ,\\
r_{n+1} & = U_{n+1} r_n\, .
\end{split}
\end{equation}

\noindent  Notice that after choosing $(\tZ_1, r_1)$, the process $(\tZ_2,
r_2), (\tZ_3, r_3), \dots$ is a scaled and translated copy of the original one,
which implies the distributional equation
\begin{equation} \label{dist-eq}
Z \stackrel{\mathcal{D}}{=}  r_1 \, Z^\prime + \tZ_1,
\end{equation}
where $Z^\prime$ is independent from $(\tZ_1,r_1)$, and has the same distribution
as $Z$. Thus,  for every $x \in [-1/2,1/2]$,
\begin{align*}
F(x) &= \int_0^{1/2} F\left( \frac{ x - y}{ 1 - 2y} \right) \, \d y +
\int_0^{1/2} F\left( \frac{ x + y}{ 1 - 2y} \right) \, \d y \\
&= (1 - 2x) \int_{-\infty} ^x \frac {F(z)}{(1-2z)^2} \, \d z +
(1 + 2x) \int_x^\infty  \frac {F(z)}{(1+2z)^2} \, \d z \\
&=(1 - 2x) \int_{-1/2} ^x \frac {F(z)}{(1-2z)^2} \, \d z + (1 + 2x)
\int_x^{1/2} \frac {F(z)}{(1+2z)^2} \, \d z + \frac{1 + 2x}{4} \, .
\end{align*}
This also shows that $F$ is continuously differentiable, and by differentiating
we arrive at
\begin{align*}
F'(x) =  -2 \int_{-1/2} ^x \frac {F(z)}{(1-2z)^2} \, \d z +
2 \int_x^{1/2} \frac{F(z)}{(1+2z)^2} \, \d z + \frac {4 x F(x)}{1 - 4x^2}
+ \frac{1}{2} \, .
\end{align*}
Once again, we derive that $F$ is twice differentiable (being the reason for
starting with the distribution function rather than the density function),
whence
\[
F''(x) = F'(x) \frac {4 x}{1 - 4x^2} \, .
\]
Taking into account that $F^\prime(x)$ is a density function on $[-1/2,1/2]$
yields that the solution is
\[
F'(x) = \frac{1}{ \pi \sqrt{(1/2 +x)(1/2-x)}} \, , \,
\]
the desired density function.
\end{proof}

\section{Further remarks}

Setting $v_i = U_1 U_2 \ldots U_i (1 - U_{i + 1})$, \eqref{zrepr} implies that
$\tZ_n = 1/2 \, \sum_{0}^{n-1} v_i  \xi_{i+1}$, hence the limit $Z$ has the
infinite series representation
\begin{equation} \label{Z-repr}
Z =  \frac{1}{2} \sum_{i=0}^{\infty} v_i  \xi_{i+1}.
\end{equation}
It is easy to check that $\sum_0^\infty v_i = 1$ almost surely, thus $v :=
(v_0, v_1, \ldots) \in \Delta$, where
\[
\Delta = \left\{ x=(x_0,x_1, \ldots ) \, : \, \sum_{i=0}^\infty x_i =1,
x_i \geq 0, i=0,1,\ldots \right\}
\]
is the infinite dimensional simplex. The construction of the random vector $v$
implies that it has the so-called GEM (Griffiths--Engen--McCloskey)
distribution with parameter 1 (see the residual allocation model in
Bertoin, 2006, p.89). This distribution appears in various contexts, such as
prime factorisation of a random integer (Hirth, 1997); and in particular, the
decreasing reordering of the GEM distribution is the so-called
Poisson--Dirichlet distribution, which is one of the most important
distributions in fragmentation theory, see Bertoin, 2006.

Using this terminology and \eqref{Z-repr}, Theorem \ref{center} can be
reformulated as: if $v = (v_0, v_1, \ldots )$ is a GEM distributed random
vector, and $\xi, \xi_1, \xi_2, \ldots$ is an iid sequence of Bernoulli random
variables such that $\PP ( \xi = 1) = 1/2 = \PP( \xi = -1)$, which is
independent of $v$, then $\sum_{i=0}^\infty v_i \xi_{i+1}$ has arcsine
distribution. This theorem was first proved by Donnelly and Tavar\'e (1987),
using the construction of the Poisson--Dirichlet distribution
by means of an inhomogeneous Poisson process. Later Hirth (1997) also
gave a proof by using the method of moments. As far as we know, our proof,
solving an integral equation based on the distributional equality
\eqref{dist-eq}, is new.

\bigskip

\noindent \textbf{Acknowledgements.} Our thanks are due to Imre B\'ar\'any for
introducing the problem to us and for illuminating discussions, and to Juan
Carlos Pardo for bringing the connection with the Poisson--Dirichlet law to our
attention. The authors are also grateful to the unknown referee for a number of
comments and suggestions that improved the paper.
While carrying out the research, the first named author held a
pleasant visiting scholarship at the Isaac Newton Institute for Mathematical
Sciences, University of Cambridge.


\begin{thebibliography}{9}


\bibitem{bertoin}
Bertoin, J. (2006) \textit{Random Fragmentation and Coagulation Processes},
Cambridge University Press.


\bibitem{bill}
Billingsley, P. (1995) \textit{Probability and Measure}, Wiley--Interscience
Publication, New York.

\bibitem{donnelly}
Donnelly, P. and Tavar\'e, S. (1987) The population genealogy of the
infinitely-many neutral alleles model. \textit{J. Math. Biol.} \textbf{25},
381--391.

\bibitem{hirth}
Hirth, U.M. (1997) Probabilistic number theory, the GEM/Poisson--Dirichlet
distribution and the arc-sine law. \textit{Combinatorics, Probability and
Computing} 6, 57--77.

\bibitem{toth}
T\'oth, B., Private communication, 2010.


\end{thebibliography}
\end{document}